\documentclass[a4paper,english,fontsize=10pt,parskip=half,abstracton]{scrartcl}
\usepackage{babel}
\usepackage[utf8]{inputenc}
\usepackage[T1]{fontenc}
\usepackage[a4paper,left=25mm,right=25mm,top=30mm,bottom=30mm]{geometry}
\usepackage{amsmath}
\usepackage{amsthm}
\usepackage{amssymb}
\usepackage{enumerate}
\usepackage{aliascnt}
\usepackage[bookmarks=false,
            pdftitle={Groups whose elements are not conjugate to their powers},
            pdfauthor={Andreas Bächle and Benjamin Sambale},
            pdfkeywords={rational groups},
            pdfstartview={FitH}]{hyperref}

\newtheorem{Thm}{Theorem} 
\newaliascnt{Lem}{Thm}
\newtheorem{Lem}[Lem]{Lemma}
\aliascntresetthe{Lem}
\newaliascnt{Prop}{Thm}
\newtheorem{Prop}[Prop]{Proposition}
\aliascntresetthe{Prop}

\allowdisplaybreaks[1]

\renewcommand{\phi}{\varphi}
\newcommand{\C}{\operatorname{C}}
\newcommand{\N}{\operatorname{N}}
\newcommand{\Z}{\operatorname{Z}}

\newcommand{\QQ}{\mathbb{Q}}

\newcommand{\FF}{\mathbb{F}}
\newcommand{\Aut}{\operatorname{Aut}}

\newcommand{\pcore}{\operatorname{O}}
\newcommand{\GL}{\operatorname{GL}}
\newcommand{\SL}{\operatorname{SL}}
\newcommand{\PSp}{\operatorname{PSp}}
\newcommand{\PSU}{\operatorname{PSU}}
\newcommand{\SU}{\operatorname{SU}}
\newcommand{\PSL}{\operatorname{PSL}}

\newcommand{\Sz}{\operatorname{Sz}}
\newcommand{\Irr}{\operatorname{Irr}}
\newcommand{\Sp}{\operatorname{Sp}}

\newcommand{\F}{\operatorname{F}}

\newcommand{\Syl}{\operatorname{Syl}}

\mathchardef\ordinarycolon\mathcode`\:  
 \mathcode`\:=\string"8000
 \begingroup \catcode`\:=\active
   \gdef:{\mathrel{\mathop\ordinarycolon}}
 \endgroup

\title{Groups whose elements are not conjugate to their powers}
\author{Andreas Bächle\footnote{Vakgroep Wiskunde, Vrije Universiteit Brussel, Pleinlaan 2, 1050 Brussels, Belgium, \href{mailto:abachle@vub.ac.be}{abachle@vub.ac.be}} \ and Benjamin Sambale\footnote{Fachbereich Mathematik, TU Kaiserslautern, 67653 Kaiserslautern, Germany, \href{mailto:sambale@mathematik.uni-kl.de}{sambale@mathematik.uni-kl.de}}}
\date{\today}

\begin{document}
\frenchspacing
\maketitle
\begin{abstract}\noindent
We call a finite group irrational if none of its elements is conjugate to a distinct power of itself. We prove that those groups are solvable and describe certain classes of these groups, where the above property is only required for $p$-elements, for $p$ from a prescribed set of primes.
\end{abstract}

\textbf{Keywords:} rational groups\\
\textbf{AMS classification:} 20D20

\section{Introduction}
All groups considered in this article are finite. A classical theme in finite group theory is the investigation of \emph{rational} groups. These are groups $G$ such that $x,y\in G$ are conjugate whenever $\langle x\rangle=\langle y\rangle$. Prominent examples of rational groups are symmetric groups.
In this paper we study groups with the opposite property, that is we call $G$ \emph{irrational} if $\langle xyx^{-1}\rangle=\langle y\rangle$ implies $xyx^{-1}=y$ for all $x,y\in G$; that is no element of $G$ is conjugate to a distinct power of itself. This property can be rephrased in terms of the character table.
Somewhat surprisingly, it seems that irrational groups have not been systematically studied in the literature so far (see \cite[Research Problem 344]{B}). 
It is easy to see that all abelian groups and all nilpotent group with squarefree exponent are irrational. Since there are many $p$-groups with exponent $p>2$, there is no hope to classify irrational groups completely. Even for $p=2$ there are already $656$ irrational groups of order $2^9$. 

To make things more accessible we define $\pi$-irrational groups for a set of primes $\pi$ in the next section. We will show that $p$-irrational groups are $p$-solvable provided $p\ge 5$. It follows easily that $2'$-irrational groups are solvable. On the other hand, we classify the simple $2$-irrational groups. Finally, we provide examples to show that the structure of irrational groups is quite unrestricted. 

\section{General results}

Our notation is fairly standard and follows \cite{Gorenstein} for instance.
For a set of primes $\pi$ we say that a finite group $G$ is \emph{$\pi$-irrational} if 
\begin{equation*}\label{irr}
\C_G(x)=\N_G(\langle x\rangle)\qquad\text{for all $\pi$-elements } x\in G.
\end{equation*}
Of course we say $p$-irrational instead of $\{p\}$-irrational. It is easy to see that $G$ is $\pi$-irrational if and only if $G$ is $p$-irrational for all $p\in\pi$. Finally, $G$ is called \emph{irrational} if $G$ is $\pi(G)$-irrational. This is easily seen to be equivalent to the definition given in the introduction.
For $g\in G$ we denote $\QQ(g):=\QQ(\chi(g):\chi\in\Irr(G))$. The $n$-th cyclotomic field is denoted by $\QQ_n$. Using \cite[Theorem~4]{Navarro03}, one can show that $G$ is $\pi$-irrational if and only if 
\begin{equation*}
\QQ(x)=\QQ_{|\langle x\rangle|}\qquad\text{for all $\pi$-elements } x\in G.
\end{equation*}
Consequently, one can read off whether a group is irrational from the character table, provided one knows the orders of conjugacy class representatives. However, the two non-abelian groups of order $p^3$ for odd primes $p$ show that the stripped character table is not sufficient for this.

Obviously, subgroups and direct products of $\pi$-irrational group are $\pi$-irrational. However, quotients of $\pi$-irrational groups are not always $\pi$-irrational. For example $\texttt{SmallGroup}(16,3)$ from the small groups library~\cite{GAP48} (which is of the form $(C_4 \times C_2) \rtimes C_2$) is irrational, but has a quotient of type $D_8$. Recall that in comparison quotients of rational groups are rational, but subgroups of rational groups in general fail to be rational. 

\begin{Lem}\label{pinormal}
If $G$ is $\pi$-irrational and $N$ is a normal $\pi'$-subgroup of $G$. Then $G/N$ is $\pi$-irrational.
\end{Lem}
\begin{proof}
Let $gN\in G/N$ be a $\pi$-element. We may assume that $g$ is a $\pi$-element. By the Schur-Zassenhaus theorem, the subgroups of order $|\langle g\rangle|$ in $\langle g\rangle N$ are conjugate under $N$. Hence, the Frattini argument yields
\[\N_{G/N}(\langle gN\rangle)\le\N_G(\langle g\rangle N)/N=\N_G(\langle g\rangle)N/N=\C_G(g)N/N\le\C_{G/N}(gN)\le\N_{G/N}(\langle gN\rangle).\qedhere\]
\end{proof}

\begin{Thm}\label{pirr}
Let $p\ge 5$ be a prime. Then every $p$-irrational group is $p$-solvable.
\end{Thm}
\begin{proof}
Let $G$ be $p$-irrational. In order to show that $G$ is $p$-solvable, we may assume that $\pcore^{2'}(G)=G=\pcore^{p'}(G)$. By \autoref{pinormal}, we can also assume that $\pcore_{p'}(G)=1$. Since $G$ is $p$-irrational, it follows that $G$ does not contain any real elements of order $p$. By \cite[Theorem~A]{DMN}, we conclude that $G$ is a direct product of certain simple groups. Hence, we may assume that $G$ itself is simple. If $G$ has cyclic Sylow $p$-subgroups, we obtain a contradiction via Burnside's transfer theorem. Consequently, \cite[Theorem~2.1]{DMN} yields $G\cong\PSL(2,p^{2f+1})$ and $p\equiv 3\pmod{4}$. In particular, the upper unitriangular matrices form an elementary abelian Sylow $p$-subgroup $P$ and $\N_G(P)$ is the set of upper triangular matrices with determinant $1$ (modulo $\langle -I_2\rangle$ of course). It follows that $G$ has only two conjugacy classes of non-trivial $p$-elements. Hence, $G$ can only be $p$-irrational if $p-1\le 2$. However this was explicitly excluded. 
\end{proof}

\begin{Thm}\label{main}
Every $2'$-irrational group has a normal Sylow $2$-subgroup. In particular, $2'$-irrational groups are solvable.
\end{Thm}
\begin{proof}
The claim follows from a more general result in \cite[Proposition~6.4]{DNT}. For the convenience of the reader we present the proof. Let $G$ be $2'$-irrational and $P\in\Syl_2(G)$. Arguing by induction on $|P|$, we may assume that $P\ne 1$. Let $x\in P$ be an involution and $g\in G$. Then $\langle x,gxg^{-1}\rangle$ is a dihedral group and since $G$ is $2'$-irrational, $\langle x,gxg^{-1}\rangle$ must be a $2$-group. Therefore, Baer's theorem~\cite[Theorem~3.8.2]{Gorenstein} implies $x\in\pcore_2(G)\ne 1$. By \autoref{pinormal}, $G/\pcore_2(G)$ is $2'$-irrational and induction shows $P/\pcore_2(G)\unlhd G/\pcore_2(G)$. It follows that $P\unlhd G$.
The last claim is a consequence of the Feit-Thompson Theorem (or of \autoref{pirr}).
\end{proof}

If $G$ is irrational (instead of $2'$-irrational), then the involutions in $G$ form an elementary abelian normal subgroup.

By \cite[Theorem~C]{DMN}, the $3$-irrational non-abelian simple groups are precisely $\PSL(2,3^{2f + 1})$ with $f\ge 1$ and the Suzuki groups $\Sz(q)$ (the only non-abelian simple groups whose order is not divisible by $3$). We now aim to describe the simple groups that are $2$-irrational.

\begin{Lem}\label{Sylow2irr} 
Let $P \in \Syl_2(G)$. Then $G$ is $2$-irrational if and only if $P$ is irrational. 
\end{Lem}

\begin{proof} 
If $G$ is $2$-irrational, then clearly $P \leq G$ is irrational. 
Assume conversely that $P$ is irrational. Let $x\in G$ be a $2$-element and $y\in\N_G(\langle x\rangle)$. Since $\Aut(\langle x\rangle)$ is a $2$-group, we may replace $y$ by its $2$-part. Then $\langle x,y\rangle$ is a $2$-group and after conjugation we have $x,y\in P$. Now the assumption gives $y\in\C_P(x)\subseteq\C_G(x)$. Hence, $G$ is $2$-irrational. 
\end{proof}

Walter's description~\cite{Walter} of all non-solvable groups with abelian Sylow $2$-subgroups provides many examples of non-solvable $2$-irrational groups. Our next lemma gives a necessary condition for $2$-irrationality.

\begin{Lem}
Let $P$ be an irrational $2$-group. Let $s$ be the number of conjugacy classes of involutions of $P$, and let $d$ be the minimal number of generators of $P$. Then $2^d\le s+1$. In particular, $\lvert\Omega(P)\rvert\ge|P/\Phi(P)|$ where $\Omega(P):=\langle x\in P:x^2=1\rangle$.
\end{Lem}
\begin{proof}
By Brauer's permutation lemma, $P$ has exactly $1+s$ real irreducible characters. On the other hand, all the inflations of the elementary abelian group $P/\Phi(P)$ are real. Now the first claim follows from $|P/\Phi(P)|=2^d$.
The second claim follows, because $\Omega(P)$ contains all involutions plus the identity and therefore $\lvert\Omega(P)\rvert\ge s+1$.
\end{proof}

\begin{Thm}\label{2simple}
Let $G$ be a finite non-abelian simple group. Then $G$ is $2$-irrational if and only if $G$ is one of the following:
\begin{itemize}
\item $\PSL(2,q)$ with $q\equiv 0,3,5\pmod{8}$,
\item $\Sz(q)$,
\item $^2G_2(q)$,
\item $J_1$.
\end{itemize}
\end{Thm}

\begin{proof} 
We use the classification of the finite simple groups. By \autoref{Sylow2irr} it suffices to decide whether a Sylow $2$-subgroup is irrational. Throughout $q$ will denote a power of a prime $p$.

Clearly $A_5$ is $2$-irrational, whereas $A_6$ is not as it contains a dihedral group of order $8$. Hence the only non-abelian simple alternating group that is $2$-irrational is $A_5 \cong \PSL(2, 5)$.

Now we consider the groups $\PSL(2, q)$ and $\PSL(3, q)$. The groups $\PSL(2, 2^f)$ are always $2$-irrational, as they have elementary abelian Sylow $2$-subgroups (note that $\PSL(2,4)\cong\PSL(2,5)$). The groups $\PSL(2, q)$, $q$ odd, have dihedral Sylow $2$-subgroups, so they happen to be $2$-irrational if and only if the Sylow $2$-subgroups are abelian, i.\,e. they have order $4$ and this happens if and only if $q$ is congruent to $3$ or $5$ modulo $8$. A group $\PSL(3, q)$ can never be $2$-irrational: Note first that the Sylow $2$-subgroups of $\PSL(3, q)$ and $\SL(3, q)$ are isomorphic. Clearly, if $q$ is even, then the upper triangular matrices contain a subgroup isomorphic to a dihedral group of order $8$. If $q$ is odd, $\SL(3, q)$ contains the elements 
\begin{equation} 
s =  \begin{pmatrix} 0 & -1 & 0 \\ 1 & 0 & 0 \\ 0 & 0 & 1 \end{pmatrix} \qquad \text{and} \qquad  t =  \begin{pmatrix} 0 & 1 & 0 \\ 1 & 0 & 0 \\ 0 & 0 & -1\end{pmatrix}\label{st}
\end{equation} 
of order $4$ and $2$, respectively. As conjugation by $t$ inverts $s$, they generate together a dihedral group of order $8$.

Next we consider general simple groups of Lie type. In the proof of \cite[Theorem~1]{BarryWard}, they exhibit in each simple group of Lie type a subgroup isomorphic to $\SL(3, q)$, $\PSL(3, q)$ or $\SL(2, q^2)$ which is not $2$-irrational, except in the following cases:
\[ A_1(q),\quad C_2(q),\quad {}^2A_2(q^2),\quad {}^2A_3(2^{2m}),\quad {}^2A_4(2^{2m}),\quad {}^2B_2(2^{2f+1}),\quad {}^2F_4(2^{2f+1}),\quad {}^2G_2(3^{2f+1}).  \]

$A_1(q) = \PSL(2, q)$: We already handled these groups.

$C_2(q) = \PSp(4, q)$: For odd $q$ it is known that the Sylow $2$-subgroups of $\Sp(4, q)$ are isomorphic to $P = Q \wr C_2$, where $Q$ is a Sylow $2$-subgroup of $\Sp(2, q) \cong \SL(2, q)$, which is quaternion of order at least $8$ (see \cite{CF}). Now $\bar{P} \in \Syl_2(\PSp(4, q))$ is isomorphic to $P/(\Z(\Sp(4, q)) \cap P)$, where $\Z(\Sp(4, q)) = \langle -I_4 \rangle$. Taking an element $x$ of order $4$ in $Q$, then it can be checked that the image of $(x, 1)$ under $Q \times Q \hookrightarrow Q \wr C_2 \twoheadrightarrow \bar{P}$ is also of order $4$ and conjugate to its inverse. So $\PSp(4, q)$, $q$ odd, cannot be $2$-irrational.\\
If $q$ is even we have that $\PSp(4, 2) \cong S_6$ which is not $2$-irrational, embeds into $\PSp(4, q)$ and we are done.

${}^2A_2(q^2) = \PSU(3, q)$: If $q$ is odd, the matrices $s$ and $t$ in \eqref{st} are matrices of $\SU(3, q)$ 
that generate a subgroup $D\cong D_8$. Note that $\Z(\SU(3, q))\cap D = 1$, so that $\PSU(3, q)$ is not $2$-irrational. \\
Now assume that $q$ is even. By \cite{Collins}, a Sylow $2$-subgroup $P$ of $\PSU(3, q)$ is given by the matrices 
\[ M(x, y) = \begin{pmatrix} 1 & 0 & 0 \\ x & 1 & 0 \\ y & x^q & 1 \end{pmatrix}, \qquad x, y \in \FF_{q^2},\ y + y^q = x^{1 + q}. \] 
Taking $x = 1$ and let $y = \zeta \in \FF_{q^2}$ be any solution of $X^q + X + 1 = 0$ (which exists since the trace map $\FF_{q^2}\to\FF_q$, $a\mapsto a+a^q$ is surjective), we get an element $A = M(1, \zeta) \in P$ of order $4$. Now let $\xi \in \FF_{q^2}$ be any root of $X^q + X + \zeta^{1+q}$ (which exists), then $B = M(\zeta, \xi) \in P$ and it is straight forward to verify that $A^B = A^{-1}$. Thus $P$ cannot be $2$-irrational.

${}^2A_n(2^{2m}) = \PSU(n+1, 2^{m})$, $n \in \{3, 4\}$: A Sylow $2$-subgroup of $\PSU(3, 2^m)$ embeds into $\PSU(4, 2^{m})$ and into $\PSU(5, 2^{m})$ (note that the center of $\SU(n, 2^m)$ always has odd order), so these groups cannot be $2$-irrational.

${}^2B_2(2^{2f+1}) = \Sz(2^{2f+1})$: The Suzuki groups are $2$-irrational: This can either be seen from their generic character table or as their Sylow $2$-subgroups of exponent $4$ are irrational, as none of the elements of order $4$ is conjugate to its inverse (see \cite{Suzuki}).

${}^2F_4(2^{2f+1})$: Using, for example, Malle's description of the maximal subgroups of ${}^2F_4(2^{2f+1})$ in \cite{Malle2F4}, we see that ${}^2F_4(2^{2f+1})$ cannot be irrational as it contains $\GL(2,3)$ as a subgroup.

${}^2G_2(3^{2f+1})$: The small Ree groups are $2$-irrational as they have elementary abelian Sylow $2$-subgroups of order $8$.

For the sporadic simple groups we consult the ATLAS~\cite{Atlas} or the character tables in GAP~\cite{GAP48}. This reveals that $J_1$ is indeed the only $2$-irrational sporadic simple group. 
\end{proof}

So all irrational groups that are Sylow $2$-subgroups of simple groups are elementary abelian or Suzuki $2$-groups. Unfortunately, not every non-abelian composition factor of a $2$-irrational group belongs to the list in \autoref{2simple}. For example, the $2$-irrational group $G:=\texttt{PerfectGroup}(2^{12}3^25,19)$ satisfies $G/\F(G)\cong A_6$ (this can be checked with GAP).

In the following we consider special families of irrational groups.

\begin{Prop}\label{collapse}\hfill
\begin{enumerate}[(i)]
 \item Every irrational metacyclic group is abelian.
 \item Every irrational supersolvable group is nilpotent.
\end{enumerate}
\end{Prop}
\begin{proof}\hfill
\begin{enumerate}[(i)]
\item Let $G$ be an irrational metacyclic group with cyclic $N = \langle x \rangle \unlhd G$ such that $G/N$ is also cyclic. Take $y \in G$, then $x^y \in \langle x \rangle$, and hence $[x, y] = 1$, so that $N \leq \Z(G)$. But then $G/\Z(G)$ is cyclic and hence $G$ is abelian.

\item Let $G$ be irrational and supersolvable. Let $N\unlhd G$ be a maximal normal subgroup. Since $G$ is solvable, $|G:N|=p$ is a prime. By induction on $|G|$, we may assume that $N$ is nilpotent. Let $g\in G\setminus N$ be a $p$-element, and let $Q$ be a Sylow $q$-subgroup of $N$ for some prime $q\ne p$. It suffices to show that $g$ acts trivially on $Q$. We may assume that $Q\ne 1$. By hypothesis, $Q$ contains a $\langle g\rangle$-stable maximal subgroup $Q_1$. By induction, $g$ acts trivially on $Q_1$. If $q=2$, then $\langle g\rangle$ clearly acts trivially on $Q/Q_1$ and therefore also on $Q$. Now let $q>2$.
By a result of Thompson (see \cite[Theorem~5.3.13]{Gorenstein}), we may assume that $Q$ has exponent $q$. The number of non-trivial subgroups which intersect $Q_1$ trivially is $(q^t-q^{t-1})/(q-1)=q^{t-1}$ where $|Q|=q^t$. Since $p\ne q$, at least one of these subgroups is normalized by $\langle g\rangle$. By irrationality, it must even by centralized. This implies the claim.\qedhere
\end{enumerate}
\end{proof}

\begin{Prop}
Let $G$ be an irrational Frobenius group with complement $K$. Then $K$ is cyclic of odd order. Conversely, every cyclic group of odd order occurs as a complement in an irrational Frobenius group. 
\end{Prop}
\begin{proof}
By \autoref{main}, $K$ has odd order. It follows from the theory of Frobenius groups that all Sylow subgroups of $K$ are cyclic and $K$ must be metacyclic (see \cite[Theorem~10.3.1]{Gorenstein}). As an irrational group, $K$ itself must be cyclic by \autoref{collapse}.

Conversely, let $K$ be any cyclic group of odd order. Let $n$ be the order of $2$ modulo $|K|$. Then $2^n\equiv 1\pmod{|K|}$ and $K$ can be embedded in a Singer cycle of $\GL(n,2)$. Since the Singer cycle acts fixed point freely, $G:=C_2^n\rtimes K$ is a Frobenius group. Obviously, $G$ is irrational. 
\end{proof}

One cannot say much about kernels in irrational Frobenius groups (apart from Thompson's Theorem that they are nilpotent). There are examples where the kernel is neither abelian nor a $p$-group. For instance take $(P\times C_{13}^2)\rtimes C_7$ where $P\in\Syl_2(\Sz(8))$ and $C_7$ acts diagonally on the direct product. 

\section{Examples}

In this section we illustrate by examples that the structure of (solvable) irrational groups can be rather wild. 
By \autoref{main}, it is natural to study irrational groups $G$ of odd order. Here it is no longer true that $G$ has a normal Sylow $p$-subgroup for some prime $p$. This can be seen from the direct product of the irrational groups $C_5^2\rtimes C_3$ and $C_3^4\rtimes C_5$. 
This can even happen for indecomposable groups as the central product of $5^{1+2}_+\rtimes C_3$ and $C_3^4\rtimes 5^{1+2}_+$ shows (the centers of order $5$ in both factors are identified).  

The Sylow $p$-subgroup of $\GL(p,p)$ of exponent $p$ shows that neither the derived length nor the nilpotency class of (nilpotent) irrational groups can be bounded. 
The following construction shows that there are irrational groups with arbitrary large Fitting length. Suppose we have given an irrational group $G$ of odd order $n$ (for example $G=C_n$). By Dirichlet's prime number theorem, there exists an odd prime $p\equiv 2\pmod{n}$. Clearly, $n$ is not divisible by $p$ and $(p-1,n)=1$. Let $V$ be any finite-dimensional faithful $\FF_pG$-module (for example the regular module). Then the semidirect product $\widehat{G}:=V\rtimes G$ satisfies $\F(\widehat{G})=V$, so that the Fitting length of $\widehat{G}$ exceeds that of $G$. 
Now let $g\in \widehat{G}$. 
If two powers of $g$ are conjugate in $\widehat{G}$, then the corresponding cosets are conjugate in $\widehat{G}/V\cong G$. Hence, if $g$ is $p$-regular, then $\N_{\widehat{G}}(\langle g\rangle)=\C_{\widehat{G}}(g)$. Now assume that $g\in V$, so that $g$ has order at most $p$. Since $(p-1,n)=1$, $g$ is not conjugate to any of its distinct powers. Thus again, $\N_{\widehat{G}}(\langle g\rangle)=\C_{\widehat{G}}(g)$ and $\widehat{G}$ is irrational. Now we can repeat the process by replacing $G$ with $\widehat{G}$.

In this way we obtain irrational groups with abelian Sylow $p$-subgroups. The following modification gives irrational groups with non-abelian Sylow $p$-subgroups. It is well-known that $G$ embeds into the symmetric group $S_n$ and $S_n$ embeds into the symplectic group $\Sp(2n,p)$ (sending a permutation matrix $M$ to $M\oplus M$). By Winter~\cite[Theorem~1]{Winter}, $G$ acts faithfully on the extraspecial group $V$ of order $p^{1+2n}$ and exponent $p$. This gives $\widehat{G}:=V\rtimes G$ with the desired properties.

From the examples above it seems that every irrational group has $p$-length $1$ for every prime $p$. This is unfortunately not true in general: The group $C_5^2\rtimes C_3$ acts faithfully on $C_3^{12}$, but the semidirect product $C_3^{12}\rtimes(C_5^2\rtimes C_3)$ is not irrational, since it contains $C_3\wr C_3$. Nevertheless there exists an irrational central extension of type $C_3^{13}.(C_5^2\rtimes C_3)$ of $3$-length $2$ (this can be checked with GAP). 

Since abelian groups are irrational, we finally describe the minimal non-abelian groups $G$ that are irrational (that is, every proper subgroup is abelian). 
First by \autoref{collapse}, we know that $G$ is not metacyclic. It turns out that this is the only constraint. Let $x \in G$ and $y \in \N_G(\langle x \rangle)$. Since $G$ is not metacyclic, $\langle x,y\rangle<G$ and $\langle x,y\rangle$ is abelian. This implies $y\in\C_G(x)$.  
Hence by the classification of minimal non-abelian groups (see e.\,g. \cite[Aufgabe~III.5.14]{Huppert} and \cite[Theorem~12.2]{SambaleBlocksBook}), $G$ is either a $p$-group of the form 
\[\langle x,y\mid x^{p^r}=y^{p^s}=[x,y]^p=[x,x,y]=[y,x,y]=1\rangle\]
with $r\ge s\ge 1$ or a group of the form $C_q^r \rtimes C_{p^s}$ for two distinct primes $p$, $q$ and $r \geq 2$, where a generator of $C_{p^s}$ acts on the elementary abelian $q$-group by a companion matrix of an irreducible divisor of $\frac{X^p-1}{X-1}$ of degree $r$.

\section*{Acknowledgment}
The work on this project started with a talk by the first author at a seminar in Kaiserslautern. We thank Alessandro Paolini for the invitation.
The first author is a postdoctoral researcher of the FWO (Research Foundation Flanders).
The second author is supported by the German Research Foundation (projects SA \mbox{2864/1-1} and SA \mbox{2864/3-1}).

\end{document}